\documentclass[12pt]{article}

\usepackage{amsmath,amssymb,amsthm}
\usepackage{tikz}
\usepackage{scalerel} 

\usepackage{graphicx,wrapfig,caption,subcaption,color}
 \usepackage{soul}  
 \usepackage[makeroom]{cancel} 
\usepackage{breqn}

\usepackage[colorlinks=true,citecolor=black,linkcolor=black,urlcolor=blue]{hyperref}

\def\nfrac#1#2{{\textstyle\frac{#1}{#2}}}
\def\dfrac#1#2{\lower0.15ex\hbox{\large$\frac{#1}{#2}$}}

\newtheorem{theorem}{Theorem}[section]
\newtheorem{lemma}[theorem]{Lemma}

\newtheorem{corollary}[theorem]{Corollary}
\numberwithin{equation}{section}

\def\kvec{{\boldsymbol{k}}}

\def\xvec{{\boldsymbol{x}}}
\def\kmax{k_{\mathrm{max}}}
\def\xmax{x_{\mathrm{max}}}
\def\dmax{d_{\mathrm{max}}}
\def\Hrk{\mathcal{H}_r(\kvec)}
\def\Hrkr{\mathcal{H}_r (k,n)}

\def\Hrkx{\mathcal{H}_r(\kvec,X)}

\def\Hrkxx{\mathcal{H}_r(\kvec-\xvec,X)}
\def\E{\mathbb{E}}
\def\P{\mathbb{P}}

\def\Mk{M(\kvec)}
\def\Mkx{M(\kvec-\xvec)}
\def\Mtkx{M_2(\kvec-\xvec)}
\def\Mtk{M_2(\kvec)}

\title{Enumerating sparse uniform hypergraphs \\ with given degree sequence and forbidden edges\thanks{Supported by the Australian Research Council grant DP140101519.}}

\author{
Haya S.~Aldosari \qquad  Catherine Greenhill \\
\small School of Mathematics and Statistics\\[-0.8ex]
\small UNSW Sydney\\[-0.8ex]
\small Sydney NSW 2052, Australia\\
\small \texttt{h.aldosari@student.unsw.edu.au} \quad \texttt{c.greenhill@unsw.edu.au}
}

\date{9 November 2018}  
\begin{document}
\maketitle

\begin{abstract}
For $n\geq 3$ and $r=r(n) \geq 3$, let $\kvec =\kvec(n)=(k_1, \ldots, k_n)$ be a
sequence of non-negative integers with sum $\Mk=\sum_{j=1}^{n} k_j$. 
We assume that $\Mk$ is divisible by $r$ for infinitely many values of $n$, and restrict
our attention to these values. 
Let $X=X(n)$ be a simple $r$-uniform hypergraph on the vertex set $V=\{v_1,v_2, \ldots, v_n\}$  
with $t$ edges. 
We denote by $\Hrk$ the set of all simple 
$r$-uniform hypergraphs on the vertex set $V$ with degree sequence $\kvec$,
and let  $\Hrkx$ be the set of all hypergraphs in $\Hrk$ which contain no edge of $X$. 
We give an asymptotic enumeration formula for the size of $\Hrkx$. 
This formula holds when $r^4 \kmax^3=o(\Mk)$,
$t\, \kmax^{3}\, =o(\Mk^2)$ and $r\,t\,\kmax^4 = o(\Mk^3)$. 
Our proof involves the switching method. 

As a corollary, we obtain an asymptotic formula
for the number of hypergraphs in $\Hrk$ which contain every edge of $X$. We apply
this result
to find asymptotic expressions for the expected number of perfect matchings and loose
Hamilton cycles in a random hypergraph in $\Hrk$ in the regular case.

\end{abstract}
\vspace*{1\baselineskip}

\section{Introduction}
Hypergraphs are increasingly used to model complex discrete systems in many areas,
including ecology~\cite{LBAA}, quantum computing~\cite{MTH}, social networks~\cite{AMPS}, 
computer science~\cite{DBRL12}, medicine~\cite{ZZKKW} and chemistry~\cite{KHT}.
However, there are relatively few asymptotic enumeration results for hypergraphs.

To describe our results we need some notation. For $n \geq 3$, let 
$\kvec=\kvec(n) =(k_1, \ldots, k_n)$ be a sequence of non-negative integers and  define
$\Mk=\sum_{i=1}^{n} k_i$.  A \emph{hypergraph} is a pair $(V,E)$ where  
$V=\{ v_1, v_2, \ldots , v_n \} $ is a set of vertices and $E$ is a multiset of 
multisubsets of $V$. The elements of $E$ are called \emph{edges}. An edge may contain a 
\emph{loop} at vertex $v\in V$ when the multiplicity of this vertex in this edge is more than one. 
A \emph{simple} hypergraph is  a hypergraph which has no loop and no repeated edge. 
A hypergraph is called $r$-\emph{uniform} if each edge contains $r$ vertices, where  $r$ is a positive 
integer. Every hypergraph in this paper has vertex set $V=\{v_1, \cdots, v_n\}$.

Now, let $X=X(n)$ be a simple $r$-uniform hypergraph on $V$ with degree sequence $\xvec$ of non-negative integers and edge set $\{ e_1, e_2, \ldots, e_t\}$. 
By a slight abuse of notation, we also write $X$ to denote its edge set. From now on, we refer to $X=\{ e_1, e_2, \ldots, e_t\}$ as the set of \emph{forbidden edges}. 

Let $\Hrk$ be the set of all simple $r$-uniform hypergraphs on $V$ with degree sequence $\kvec$, and $\Hrkx$ be the set of all hypergraphs in $\Hrk$ which contain no edge of $X$. 
The aim of this paper is to estimate the size of $\Hrkx$. In other words, we find an asymptotic expression for the number of simple $r$-uniform
hypergraphs with given degree sequence $\kvec$ which avoid the forbidden edges. 
Throughout this paper, we assume that $r$ divides $\Mk$ for infinitely many values of $n$ and take $n$ to infinity along these values. 

For a positive integer $k$,  let $\Hrkr$ denote the regular case of $\Hrk$, where all the vertices have the same degree $k$. A hypergraph chosen uniformly at random from $\Hrk$ will be referred to as a \emph{random hypergraph from} $\Hrk$, and similarly for $\Hrkr$.
We write $(a)_b$  for the falling factorial $a(a-1) \cdots (a-b+1)$, 
and define  $\Mtk = \sum_{i=1}^{n} (k_i)_2$.

\begin{theorem}\label{HkX} 
For $n \geq 3$, suppose that $r=r(n) \geq 3$. Let $\kvec=\kvec(n)= (k_1,\ldots, k_n)$ be a sequence of non-negative integers with maximum degree $\kmax$ and  sum $\Mk$. We assume that $r$ divides $\Mk$ for infinitely many values of $n$.
Let $X=X(n)$ be a given simple $r$-uniform hypergraph with degree sequence $\xvec$  
and with $t$ edges. Suppose that $r^4 \, \kmax^{3}=o\left(\Mk\right)$ and 
$\rho=o(1)$, where
\[
\rho=\frac{t\, \kmax^3}{\Mk^2} \, +\, \frac{ r\,t\, \kmax^4}{\Mk^3}.
\]
Then the probability that a random hypergraph from $\Hrk$ contains no edge of $X$ is $\exp\left(O(\rho)\right)$. Therefore, the number of simple $r$-uniform hypergraphs with degree sequence $\kvec$ containing no edge of $X$ is
\begin{align*}
|\Hrkx| =\frac{ \Mk !}{(\Mk/r)!  r!^{\Mk/r} \, \prod_{i=1}^{n} k_i !} \, \exp \left(
 \frac{ -(r-1) \Mtk}{ 2\Mk} + O\left( \frac{r^4 \, \kmax^3}{\Mk} + \rho \right)
\right).
\end{align*}
\end{theorem}

Our proof of this result uses the switching method, as described in Section~\ref{switch}.
Using more complicated switchings, it should be possible to prove an asymptotic enumeration
formula for an extended range of parameters, allowing a small but non-vanishing expected number of forbidden edges.
Theorem~\ref{HkX} is sufficient for our purposes, so we leave this extension for 
future work.

As a consequence of Theorem~\ref{HkX}, we can obtain an asymptotic formula for the
probability that a random element of $\Hrk$ contains all edges of $X$. 
This will be useful in applications, as illustrated in Section~\ref{s:applications}.

\begin{corollary}\label{PROBincEXC}
For $n\geq 3$ and $r=r(n) \geq 3$, let $\kvec$ and $\kmax$ be defined as above. 
Let $X=X(n)$ be a given simple $r$-uniform hypergraph with degree sequence $\xvec$ 
and with $t$ edges,
where $x_i \leq k_i$ for all $i=1,2,\ldots,n$. 
Define \[\beta =\frac{r^4 \kmax^3}{\Mkx} \, + \,\frac{t \, \kmax^3}{\Mkx^2} \, +\, 
  \frac{r\,t\, \kmax^4}{\Mkx^3}, \]
and assume that $\beta=o(1)$. Then 
the probability that a random hypergraph from $\Hrk$ contains every edge of $X$ is
\[
\frac{(\Mk/r)_{t} \, r!^{t}\, \prod_{i=1}^{n} (k_i)_{x_i} }{(\Mk)_{rt}}\,
  \exp\left(\frac{r-1}{2}\,\left(\frac{\Mtk}{\Mk}-  \frac{\Mtkx}{\Mkx} \right)+ O\left( \beta \right)\right). \\
 \]
\end{corollary}

\begin{proof}
For a given $r$-uniform hypergraph $X$,  the number of hypergraphs with degree sequence $\kvec$ which contain every edge of $X$  is equal to the number of hypergraphs with degree sequence $\kvec-\xvec$ which contain no edge of $X$. Therefore, the probability that a random hypergraph  $\mathcal{H} \in \Hrk$  contains $X$ is 
\[  \P(X\subseteq \mathcal{H})=\frac{|\Hrkxx|}{|\Hrk|} .\] \\
This probability can be computed using Theorem~\ref{HkX}, leading to the stated expression with
error term given by 
 \begin{align*}
 &\frac{r^4 \kmax^3}{\Mk} \,  +\, \frac{r^4 \dmax^3}{\Mkx} \, + \, \frac{t\,\dmax^3}{\Mkx^2} \, +\, 
\frac{ r\, t\, \, \dmax^4}{\Mkx^3}  
  =O(\beta),
 \end{align*}
 where $\dmax=\max \{ k_j -x_j\,: j=1, \ldots, n\}$.
\end{proof}
 \subsection{History}
 
This section describes some previous studies on enumeration of some classes of graphs and hypergraphs with various restrictions.

For sparse graphs,  McKay~\cite{Mck85} established the first result on the number of simple graphs with given degree sequence $\kvec$ avoiding a certain set of  edges. His result holds when $\kmax \left(\kmax +\xmax\right) =o(M)$, 
where $\xmax$ is the maximum degree in $\xvec$. There are also other studies on the asymptotic enumeration of graphs with given degrees, see for example~\cite{BarHar13,LieWor17,MckWor91}, but less work has been done on forbidding a given set of edges.

In the dense regime, McKay~\cite{Mck11} found an asymptotic enumeration formula for simple graphs with given degree 
sequence $\kvec$ avoiding a certain set of edges $X$. This formula holds when the average degree is roughly linear,
the degree sequence is close to uniform and $|X|$ is roughly linear in $n$: see~\cite{Mck11} for more details.
His proof uses the saddle-point method.

There are few asymptotic enumeration results for simple $r$-uniform hypergraphs. 
The regular case was considered
by Dudek et al.~\cite{DFRS13} in the sparse regime and by Kuperberg et al.~\cite{KLP17} in the dense regime. 

For uniform hypergraphs with irregular degree sequences, Blinovsky and Greenhill~\cite{BliGre16a} 
estimated the cardinality of $\Hrk$ when $\kmax^3 = o(\Mk)$, treating $r$ as constant. This result has been extended 
in~\cite[Corollary~2.3]{BliGre16b} to consider slowly-growing $r$. Our result (Theorem~\ref{HkX}) relies on 
this formula, restated below.

 \begin{lemma}\emph{\cite[Corollary 2.3]{BliGre16b}}\,  \label{BliGre}
For $n\geq 3$, let $r=r(n), \,\kvec=\kvec(n), \text{ and } \Mk$ be defined as in Theorem~\ref{HkX}. 
Suppose that $\Mk \to\infty$ and $r^4\kmax^3 = o(\Mk)$ as $n\rightarrow \infty$. Then
\[
|\Hrk| =   \frac{ \Mk !}{(\Mk/r)! r!^{\Mk/r} \prod_{i=1}^{n} k_i !} \exp \left(
 \frac{ -(r-1) \Mtk}{ 2\Mk} +O\left(\frac{r^4 \kmax^{3}}{ \Mk} \right) \right).
\]
\end{lemma} 
In recent work, Espuny D{\'i}az et al.~\cite{EFKO18} 
estimated the probability that a random hypergraph from $\Hrkr$ contains a fixed set of edges $X$. They also gave a formula for the expected number of copies of $X$ in a random hypergraph from $\Hrkr$,  
when $r\geq 3$ is fixed and $k$ satisfies $k = \omega(1)$ and $k=o(n^{r-1})$.
  To the best of our knowledge, there are no other results on asymptotic enumeration of 
uniform hypergraphs with given degree sequence $\kvec$ and forbidden edges: 
in particular, there are no prior results when $\kvec$ is irregular or $\kmax=O(1)$.

\subsection{Structure of our argument}

We now outline our argument, and then describe the structure of the paper. 

Recall that $X=\{e_1,\ldots, e_t\}$.
For $i=1,\ldots, t$, let $\mathcal{F}_i \subseteq \Hrk$ be the set of hypergraphs in $\Hrk$ which contain the 
edge $e_i$.  Define $\mathcal{F}=\cup_{i=1}^{t} \mathcal{F}_i $ and observe that $\mathcal{F}^c$ is the set of hypergraphs in $\Hrk$ which contain no edges of $X$. Define $\xi_i = |\mathcal{F}_i|/|\mathcal{F}_{i}^{c}|$ for $i=1,\ldots, t$. Then, for a random hypergraph from $\Hrk$,
\begin{align*}
\P(\mathcal{F}) \leq \sum_{i=1}^{t} \P(\mathcal{F}_i) \leq \sum_{i=1}^{t} \xi_i .
\end{align*}
If $\sum_{i=1}^{t}  \xi_i  =o(1)$ then
\begin{align*}
1\geq \P(\mathcal{F}^c) \geq 1-\sum_{i=1}^{t} \xi_i =1-o(1),
\end{align*}
and hence
\begin{align}\label{Rt}
\P(\mathcal{F}^c)=\exp\left(- \sum_{i=1}^t \xi_i + O\left(\biggl(\sum_{i=1}^{t} \xi_i \biggr)^2\right)\right).
\end{align}

In Section~\ref{switch} we use the method of switchings to obtain an upper bound on $\xi_i$. 
The proof of Theorem~\ref{HkX} is completed in Section~\ref{proofth1}, using Lemma~\ref{BliGre}. 
In Section~\ref{s:applications} we present two applications of Corollary~\ref{PROBincEXC}, 
giving asymptotic expressions for
the expected number of perfect matchings
and the expected number of loose Hamilton cycles in a random hypergraph in $\Hrkr$,
when $r^4 k^2=o(n)$.

\section{The switchings}\label{switch}
Let $e_i \in X$ be given. 
We will now define and analyse a switching operation in order to obtain an upper bound on $\xi_i$.

\subsection{Forward switching}

Suppose that $G^*$ is a hypergraph in $\mathcal{F}_i$. 
Define $S^*=S^*(G^*,i) $ to be the set of all  $6$-tuples $(z_1,z_2,y_1,y_2,f_1,f_2)$ defined as follows:
\begin{itemize}
\item $z_1, z_2, y_1,y_2 $ are distinct vertices from $V$,
\item $e_i, f_1, f_2 $ are distinct edges of $G^*$, and
\item $z_1, z_2 \in e_i $ and $y_j \in f_j $ for $j= 1,2$.
\end{itemize}
Let $G$ be a hypergraph resulting from a forward switching operation on $G^*$  determined by the 
$6$-tuple $(z_1,z_2,y_1,y_2,f_1,f_2)$.  That is,
  \[
  G = \left( G^* \setminus \{e_i, f_1, f_2\} \right) \cup \{g, g_1, g_2 \},\]
   where $g = \left( e_i \setminus  \{z_1,z_2\} \right) \cup  \{y_1,y_2\}$  and 
$g_j = \left( f_j \setminus  \{y_j\}\right) \cup  \{z_j\}$, for $j =1,2$.
This switching is illustrated in Figure~\ref{switching}, following the arrow from left to right. 

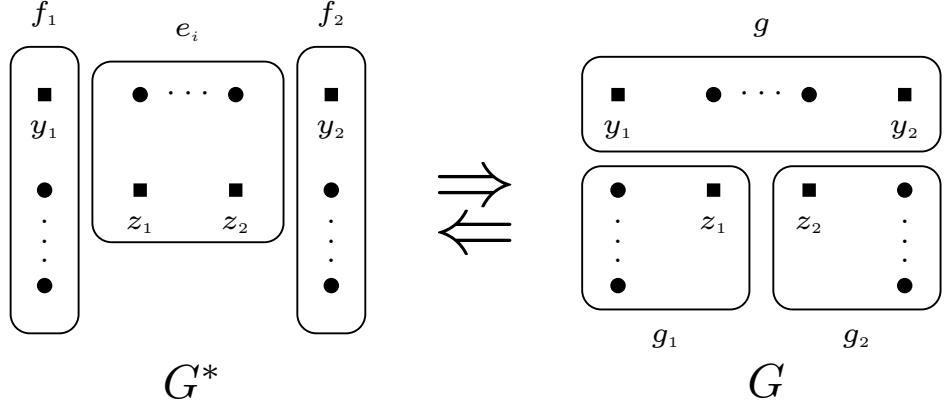
\begin{figure}[ht!]
\begin{center}
\resizebox{0.8\linewidth}{!}{
\begin{tikzpicture}[squarep/.style={draw,rectangle,inner sep=1.2pt,fill=black},
dot/.style={draw, circle, inner sep=1pt, fill=black},
column/.style={rectangle, draw, 
    text height=40pt, text centered, rounded corners, minimum height=60pt, minimum width=.5cm},
square/.style={rectangle, draw, 
    text height=30pt, text centered, rounded corners, minimum height=30pt, minimum width=40pt},
    squarer/.style={rectangle, draw, 
    text height=0pt, text centered, rounded corners, minimum height=30pt, minimum width=35pt},
    row/.style={rectangle, draw, 
    text height=12pt, text centered, rounded corners, minimum height=.5cm, minimum width=75pt}]

 \node[dot]  (u1) {};
  \node[squarep,draw,label=below:\tiny{$z_{\scaleto{1}{2pt}}$}] at ([yshift=-20pt]u1)  (z1) {};
  \node[dot] at ([xshift=20pt,yshift=20pt]z1) (u2){};
  \node[label=below:\tiny{$\cdots$}] at ([xshift=10pt,yshift=10pt]u1) {};
 \node[squarep,label=below:\tiny{$z_{\scaleto{2}{2pt}}$}] at ([xshift=20pt]z1)  (z2) {};

\node[squarep,label=below:\tiny{$y_{\scaleto{1}{2pt}}$}] at ([xshift=-20pt,yshift=20pt]z1) (y1) {};

\node[squarep,label=below:\tiny{$y_{\scaleto{2}{2pt}}$}] at ([xshift=20pt]u2) (y2) {};

\node[dot] at ([yshift=-20pt]y1) (f1u1) {};
\node[label=below:\tiny{$\vdots$}] at ([yshift=8pt]f1u1) {};
\node[dot] at ([yshift=-20pt]f1u1) (f1u2) {};

\node[dot] at ([yshift=-20pt]y2) (f2u1) {};
\node[label=below:\tiny{$\vdots$}] at ([yshift=8pt]f2u1) {};
\node[dot] at ([yshift=-20pt]f2u1) (f2u2) {};
\node [column,label={above:\tiny{$f_{\scaleto{1}{2pt}}$}}] at ([yshift=-20pt]y1) (f1) {};
\node [column,label={above:\tiny{$f_{\scaleto{2}{2pt}}$}}] at ([yshift=-20pt]y2) (f2) {};
\node [square,label={above:\tiny{$e_{\scaleto{i}{2pt}}$}}] at ([xshift=10pt,yshift=-12pt]u1) (ei) {};

 \node[dot]  at([xshift=120pt]u1) (u1r) {};
 \node[squarep,draw,label=below:\tiny{$z_{\scaleto{1}{2pt}}$}] at ([yshift=-20pt]u1r)  (z1r) {};
  \node[dot] at ([xshift=20pt,yshift=20pt]z1r) (u2r){};
  \node[squarep,label=below:\tiny{$z_{\scaleto{2}{2pt}}$}] at ([xshift=20pt]z1r)  (z2r) {};
  \node[label=below:\tiny{$\cdots$}] at ([xshift=10pt,yshift=10pt]u1r) {};
\node[squarep,label=below:\tiny{$y_{\scaleto{1}{2pt}}$}] at ([xshift=-20pt,yshift=20pt]z1r) (y1r) {};

\node[squarep,label=below:\tiny{$y_{\scaleto{2}{2pt}}$}] at ([xshift=20pt]u2r) (y2r) {};
  \node[dot] at ([yshift=-20pt]y1r) (f1u1r) {};
\node[label=below:\tiny{$\vdots$}] at ([yshift=8pt]f1u1r) {};
\node[dot] at ([yshift=-20pt]f1u1r) (f1u2r) {};

\node[dot] at ([yshift=-20pt]y2r) (f2u1r) {};
\node[label=below:\tiny{$\vdots$}] at ([yshift=8pt]f2u1r) {};
\node[dot] at ([yshift=-20pt]f2u1r) (f2u2r) {};
\node [row,label={above :\tiny{$g$}}] at ([xshift=30pt,yshift=-2pt]y1r) (g) {};
\node [squarer,label={below:\tiny{$g_{\scaleto{1}{2pt}}$}}] at ([xshift=-10pt,yshift=-10pt]z1r) (g1) {};
\node [squarer,label={below:\tiny{$g_{\scaleto{2}{2pt}}$}}] at ([xshift=10pt,yshift=-10pt]z2r) (g2) {};
\node at ([xshift=30pt]f2u1)  {\Large{$\Rightarrow$}};
\node at ([xshift=30pt,yshift=-10pt]f2u1)  {\Large{$\Leftarrow$}};
\node at ([xshift=11pt,yshift=-40pt]z1)  {\footnotesize{$G^*$}};
\node at ([xshift=11pt,yshift=-40pt]z1r)  {\footnotesize $G$};
\end{tikzpicture}
}
\caption{The forward and reverse switchings}
\label{switching}
\end{center}
\end{figure}

We say that the forward switching given by $(z_1,z_2,y_1,y_2,f_1,f_2)$ is \emph{legal}
if $G\in \mathcal{F}_{i}^{c}$, otherwise it is illegal.
The next lemma describes illegal forward switchings from $G^*$.


\begin{lemma}\label{illegalFswitch} 
Let $G^* \in \mathcal{F}_i$. Suppose that the 
$6$-tuple $(z_1,z_2,y_1, y_2, f_1, f_2) \in S^*$ results in an illegal forward switching from $G^*$. 
Then at least one of the following holds:
\begin{itemize}
\item[\emph{(I)}]At least one of $z_j, y_j$ belongs to both edges $e_i$ and  $f_j$, for some $j\in \{1,2\}$.

\item[\emph{(II)}] There is an edge $e \in  G^* \setminus \{e_i, f_1,f_2 \} $ such that either

\begin{itemize}
\item[\emph{(a)}] $e \cap e_i = e_i\setminus \{z_1, z_2 \}$ and  $e \cap f_j= \{y_j\}$  for $j=1,2$,  or

\item[\emph{(b)}]$e \cap f_j= f_j\setminus \{y_j\}$ and $e\cap e_i= \{ z_j\}$, for some $j\in \{1,2\}$. 

\end{itemize}
\item[\emph{(III)}] For some $j\in \{1,2\}$, $f_j \setminus \{y_j\} =e_i \setminus \{z_j\}$.
\end{itemize}  
\end{lemma}

\begin{proof}
Suppose that $(z_1,z_2,y_1,y_2,f_1,f_2)$ is a $6$-tuple in $S^*$ which gives  an illegal switching on $G^* \in \mathcal{F}_i$. This means that the resulting hypergraph $G$ does not belong to $\mathcal{F}_{i}^{c}$. Then we have at least one of the following situations:
\begin{itemize}
\item[$\circ$] {\bf $G$ contains a loop.}\ This implies that at least one new  loop has been created accidentally 
at one of the vertices $z_j, y_j$ for some $j\in \{1,2\}$. If $g_j$ contains  a loop at $z_j$ for some $j \in \{1,2\}$, then we have $z_j \in f_j \cap e_i$ in $G^*$. Therefore, (I) holds. Similarly, if $g$ has a loop at  $y_j$ for some $j \in \{1,2\}$ then $y_j \in f_j \cap e_i$ in $G^*$, so (I) holds.

\item[$\circ$] {\bf $G$ contains a repeated edge.}\ Then the repeated edge must involve one of the new edges $g, g_1, g_2$, since $G^* \in \mathcal{F}_i$ is simple. Suppose that $g$ has multiplicity greater than one in $G$. Then $g$ also belongs to 
$G^* \setminus \{e_i, f_1, f_2\}$, as an edge of multiplicity~1. Hence, $g\setminus \{y_1,y_2\} =e_i \setminus \{z_1, z_2\}$. In addition, $g$ intersects  both $f_1, f_2$ in $y_1, y_2$, respectively. Hence (II)(a) holds. 
Similarly, if $g_j$ is a multiple edge in $G$ for some $j\in \{1,2\}$ then $g_j$ also belongs to 
$G^* \setminus \{e_i, f_1, f_2 \}$, and (II)(b) holds.

\item[$\circ$] {\bf $G$ contains the edge $e_i$.}\   Since $G^*$ is simple and $z_1, z_2, y_1, y_2$ are distinct vertices, 
either $g_1 = e_i$ or $g_2 = e_i$. Then $g_j \setminus \{z_j\}$ is the same set as $e_i \setminus \{z_j\}$, for some $j \in \{1,2\}$. From the definition of $g_j$ we also have $g_j \setminus \{z_j\}=f_j \setminus \{y_j\} $. Therefore (III) holds.
\end{itemize}
This completes the proof.
\end{proof}

Next, we analyse forward switchings.

\begin{lemma}\label{eq1}
Let $G^*$ be a hypergraph in $\mathcal{F}_i$ and let $S^*=S^*(G^*,i)$ be the set of $6$-tuples $(z_1,z_2,y_1,y_2,f_1,f_2)$ 
defined earlier. 
If $r^4 \, \kmax^3=o(\Mk)$ then the number of 6-tuples in $S^*$ which determine a  legal switching is 
 \begin{align*}
 r(r-1) \Mk^2 \, \left(  1+ O \left( \frac{ \kmax^2 +  r \,\kmax }{\Mk}  \right) \right).
 \end{align*}
\end{lemma}
\begin{proof} 
From the definition of $S^*$, it is obvious that the number of 6-tuples which determine a legal forward switching on $G^*$ 
is bounded above by  $|S^*|$.  To find a lower bound on this number, we will subtract from $|S^*|$ an 
estimate for the number of 6-tuples which result in an illegal switching. 
These illegal 6-tuples are described in Lemma~\ref{illegalFswitch}. 

First, we will find an asymptotic expression for $|S^*|$. There are $r (r-1)$ choices for $(z_1,z_2)$, as a pair of
distinct vertices from $e_i$, and at most $\Mk^2$ choices for $(y_1,y_2,f_1,f_2)$. Therefore, 
\[
 |S^*| \leq r(r-1) \Mk^2.
 \]
Now we find a lower bound for $|S^*|$.  First we need to choose an edge $f_1\neq e_i$ and a vertex $y_1\in f_1$ 
such that $y_1\not\in\{ z_1,z_2\}$. The number of ways to choose $(y_1,f_1)$ is at least
\[ r \left(\frac{\Mk}{r} -1\right)- 2\kmax = \Mk \left( 1+ O\left(\frac{r+\kmax}{\Mk}\right) \right). \]
Next, the number of choices for $(y_2,f_2)$ such that $f_2 \notin \{e_i,f_1\}$,\,  
$y_2\in f_2$ and $y_2 \notin \{z_1,z_2,y_1\}$ is at least 
$ \Mk \left(1+ O\left(\nfrac{r+\kmax}{\Mk} \right)\right)$.

Combining the bounds of $|S^*|$, we have
\begin{align}
|S^*| &= r (r-1) \Mk^2 \left( 1+ O \left( \frac{r+\kmax}{\Mk} \right) \right).
\label{S-star-def}
\end{align}
Now, we estimate an upper bound for the number of 6-tuples in $S^*$ which satisfy some property in Lemma~\ref{illegalFswitch}.

For (I), suppose that $y_j  \in e_i \cap f_j$ for some $j\in \{1,2 \}$. There are $r(r-1)$  ways to choose $(z_1, z_2)$ and at most  $(r-2)\, \kmax \, \Mk$ choices for  $(y_1,y_2,f_1, f_2)$ satisfyings this condition. 
Similarly, if $z_j \in e_i \cap f_j $  for some $j\in \{1,2 \}$ then we have $r (r-1)$ choices for $(z_1,z_2)$ and at most 
$(r-1) \, \kmax \, \Mk$ choices for $(y_1,y_2,f_1, f_2)$. Therefore, the number of $6$-tuples in $S^*$ satisfying  
(I)  is at most 
\[ 2r(r-1)^2\kmax \Mk.\]

For (II)(a), suppose that there exists an edge $e \in G^* \setminus \{e_i, f_1,f_2 \} $ such that 
$e\cap e_i=e_i \setminus \{z_1,z_2\}$ and $e \cap f_j = \{y_j\}$ for $j= 1,2$. 
There are $r(r-1)$ choices for $(z_1,z_2)$ as distinct vertices in $e_i$.
Then there are at most $\kmax$ choices for $e$ and two ways to choose $(y_1,y_2)$, as these are the two vertices  in $e\setminus e_i$. 
We also have at most $\kmax^{2}$ choices  for $(f_1,f_2)$ as incident edges for $y_1, y_2$, respectively. 
Therefore, the number of 6-tuples in $S^*$ satisfying (II)(a) is at most
\[ 2r(r-1) \, \kmax^{3}.\]
 
For (II)(b),  suppose that there exists an edge $e \in G^* \setminus \{e_i, f_1,f_2 \}$ such that 
$e\cap f_j = f_j \setminus \{y_j\}$ and $e \cap e_i=z_j$ for some $j \in \{ 1,2\}$. 
Then the number of choices for the $6$-tuple satisfying (II)(b) is at most
\[ r(r-1) \, \kmax^2 \Mk.\]
  
For (III), suppose that  $f_j\setminus \{y_j\} =e_i \setminus \{z_j\} $ for some $j \in\{1,2\}$. 
Arguing as above, the number of 6-tuples in $S^*$ satisfying this condition is at most
\[ r(r-1)\,\kmax\, \Mk.\] 
 
Combining these cases shows that the number of 6-tuples in $S^*$ which give rise to an illegal switching is at most
\begin{align*} 
& 2r(r-1)^2\kmax \Mk +2r(r-1)\, \kmax^{3} + \, r(r-1)(\kmax +1 ) \, \kmax \Mk  \\
 &=r (r-1) \Mk^2 \, O\left(\frac{ (r+ \kmax) \kmax}{\Mk} +\frac{\kmax^3}{\Mk^2}\right)\\
  &=r (r-1) \Mk^2 \, O\left(\frac{ \kmax^2+ r \, \kmax}{\Mk} \right).
 \end{align*}
Subtracting this from (\ref{S-star-def})  completes the proof.
\end{proof}

\subsection{Reverse switching}

Let  $G \in \mathcal{F}^{c}_{i}$ be chosen at random and $S=S(G,i)$ be the set of all 
$6$-tuples $(z_1,z_2,y_1,y_2,g_1,g_2)$ defined as follows:
\begin{itemize}
\item $z_1, z_2, y_1,y_2 $ are distinct vertices in $V,$
\item  $g_1, g_2 $ are distinct edges of $G$,
\item $g_j \cap e_i =\{z_j \}$ for $j=1,2$, and
\item there is an edge $g \in G$ which contains $y_1, y_2 $ such that $g \cap e_i =e_i \setminus \{z_1,z_2 \}$.
\end{itemize}
A reverse switching on  $G$ operating by the $6$-tuple $(z_1,z_2,y_1,y_2,g_1,g_2) $ results in  a hypergraph $G^*$ defined by
\[
G^* = \left( G \setminus \{g, g_1, g_2\} \right) \cup \{e_i, f_1, f_2 \},
\]
 where $ f_j = \left( g_j \setminus  \{z_j\}\right) \cup  \{y_j\}, \text{ for } j =1,2$.
This reverse switching is illustrated in Figure~\ref{switching} by reversing the arrow. We say that the reverse switching is 
\emph{legal} if $G^* \in \mathcal{F}_i$.

Every 6-tuple which gives rise to a legal reverse switching belongs to $S$. 
Therefore, it is sufficient to obtain an upper bound on $|S|$ in order to upper-bound the number of legal reverse switchings. 
Since $z_1 , z_2 \in e_i$, there are at most $r (r-1)\, \kmax^{2} $ choices for $(z_1,z_2,g_1, g_2) $ such that $z_j \in g_j$ for $j=1,2$. Also we have at most $2 \kmax$ choices for $(y_1,y_2)$ such that these vertices belong to an edge  which intersects with $e_i$ in exactly $r-2$ vertices. Therefore, the number of legal reverse switchings which can be performed on $G$ is at most 
\begin{align}\label{upperbound}
2r(r-1) \kmax^{3} .
\end{align}
Now, we can complete the proof of our main result.

\subsection{Proof of Theorem~\ref{HkX}}\label{proofth1}
\begin{proof}
We conclude from  Lemma~\ref{eq1} and~(\ref{upperbound}) that

\begin{align}\label{xi}
\xi_i =\, \frac{|\mathcal{F}_i|}{|\mathcal{F}_{i}^{c}|} &\leq \frac{2 r\, (r-1) \, \kmax^{3}}{r(r-1) \Mk^2 \, 
  \left(  1+ O \left( \left(\kmax^2 + r \, \kmax\right)/\Mk   \right)\right)} \notag \\[6pt]
&=O \left( \frac{\kmax^3}{\Mk^2} \, +\, \frac{ r \, \kmax^4}{\Mk^3} \right).
\end{align}
The assumptions of Theorem~\ref{HkX} imply that 
\[  \sum_{i=1}^{t} \xi_i 
= O \left( \frac{t\, \kmax^3}{\Mk^2} \, +\, \frac{ r\,t\,\kmax^4}{\Mk^3} \right)
=o(1).\]
Therefore, by (\ref{Rt}) and (\ref{xi}),
\begin{align}\label{final Rt}
\P(\mathcal{F}^c)
 &= \exp\left(
O \left( \frac{t\, \kmax^3}{\Mk^2} \, +\, \frac{ r\,t\,\kmax^4}{\Mk^3} \right)
  \right).
\end{align}
We complete the proof by multiplying (\ref{final Rt}) by the value of $|\Hrk|$ given by Lemma~\ref{BliGre}. 
\end{proof}

\section{Some applications}\label{s:applications}

Corollary~\ref{PROBincEXC}  can be used to estimate the expected number of substructures of a random element of $\Hrk$.
To illustrate this,   
we estimate the number of perfect matchings and loose Hamilton cycles in the regular setting.

\subsection{Perfect matchings}
 
 For $n \geq 3$, suppose that $r=r(n) \geq 3$ is a factor of $n$. When $n$ is divisible by $r$, a set of $n/r$ edges of $H \in \Hrkr$ which covers  all the vertices of $H$ exactly once is called a \emph{perfect matching} in $H$.  Let $\mathcal{H}$ be chosen uniformly at random from $\Hrkr$, and let $Z$ be the number of perfect matchings in  $\mathcal{H}$. 
The number of perfect matchings in $k$-regular $r$-uniform hypergraphs was analysed by Cooper et al.~\cite{CFMR96} when 
$r\geq 3$ and $k\geq 2$ are fixed integers.  
Combining~\cite[Lemma 3.1]{CFMR96} with~\cite[(6.18)]{CFMR96} implies that under these conditions,
the expected number of perfect matchings in a 
random hypergraph in $\Hrkr$ is
 \begin{align}\label{EZpm}
 \E(Z) = (1+o(1))\, e^{(r-1)/2} \,\sqrt{r} \left(k \left(\frac{k-1}{k}\right)^{(r-1)(k-1)} \right)^{n/r}.
  \end{align}
In the following corollary, we show that the same formula holds when $k$ and $r$  grow 
sufficiently slowly as $n \rightarrow \infty$. 
 \begin{corollary}\label{EZ}
 For a positive integer $n \geq 3$, let $r =r(n) \geq 3$ be such that $r$ divides $n$ for infinitely many values of 
$n$. Let $k=k(n) \geq 2$ and let $Z$ denote the number of perfect matchings in a hypergraph chosen randomly from 
$\Hrkr$. Then, when $r^4\, k^2=o(n)$,   
 \[
 \E(Z) = \sqrt{r} \left(k \left(\frac{k-1}{k}\right)^{(r-1)(k-1)} \right)^{n/r} \exp \left(\frac{r-1}{2} + O\left(\frac{r^4 k^2 }{n} \right)\right).
 \]
 \end{corollary}
 \begin{proof}
 Let $\mathcal{H}$ be chosen uniformly at random from $\Hrkr$. 
Then
\begin{align}\label{E}
\E(Z)&= \sum_{X} \P(X\subseteq  \mathcal{H}), 
\end{align}
where the sum is over all possible perfect matchings $X$. 
Now let $X$ be a fixed perfect matching with $t=n/r$ edges.
Since $X$ is has degree sequence $\boldsymbol{x}=(1,1,\ldots, 1)$, we have
$\kvec-\xvec = (k-1,\ldots, k-1)$ and
 \[ \Mkx = (k-1)n,\qquad \Mtkx =  (k-1) (k-2) n.
 \]
Then, by Corollary~\ref{PROBincEXC}, 
\begin{align*}
\P(X\subseteq \mathcal{H})
    &=\frac{(kn/r)_{t} \, r!^{t}\, k^{n}  }{(kn)_{n}}  \exp\left(\frac{r-1}{2} \left(\frac{k\,(k-1)\,n}{k\,n}- \, \frac{(k-1)\,(k-2)\,n}{(k-1)\,n} \right) +O\left(\beta\right) \right),
\end{align*}
where
\[
\beta = \frac{r^4 k^3 }{(k-1)n} \, + \,\frac{t\,k^3}{(k-1)^2 n^2} \, +\, \frac{r\,t\,k^4}{(k-1)^3 n^3} =
O \left(\frac{r^4 \, k^2}{n}  \right).
\]
The number of  perfect matchings in  the complete $r$-uniform hypergraph on $n$ vertices is
\[ \frac{n!}{(n/r)!\,\, r!^{n/r}}. \]
Hence by symmetry, using (\ref{E}), we obtain
\begin{align*}
\E(Z)&= \frac{n!}{(n/r)!\,\, r!^{n/r}} \frac{(\dfrac{kn}{r})! \, (kn-n)!\, r!^{t}\, k^{n}  }{(\dfrac{kn}{r}-t)!(kn)!} \exp\left(\frac{r-1}{2}   +O\left(\frac{r^4 k^2}{n}\right) \right).
\end{align*}
The factorial terms in this formula can be expanded by applying Stirling's formula, giving error term $O(r/n)$
which is absorbed by the  error term. This completes the proof.
\end{proof}

\subsection{Loose Hamilton cycles}
Suppose that $r-1$ divides $n$. A \emph{loose Hamilton cycle} is a set of $t=\nfrac{n}{r-1}$ edges which can be labelled as $e_0, \ldots, e_{t-1}$ such that,  for some ordering $v_0, \ldots, v_{n-1}$ of the vertices, 
\[
e_i = \{ v_{i(r-1)}, v_{i(r-1)+1}, \ldots, v_{i(r-1)+(i-1)}\}.
\]
Let $C$ be a loose Hamilton cycle with edge set $\{ e_1, \ldots, e_t\}$. The degree sequence of $C$ is $\xvec =(x_1, \ldots, x_n)$ where $t$ vertices have degree $2$ and all remaining vertices have degree $1$.

When $r\geq 3$ and $k\geq 2$ are fixed integers, the number of Hamilton cycles in a random element of $\Hrkr$ has been 
studied by Altman et al in~\cite{AGIR17}.  
Let $\mathcal{H}$ be chosen uniformly at random from $\Hrkr$ and
let $Y$ be the number of loose Hamilton cycles in $\mathcal{H}$. 
Altman et~al.\ conjectured that 
\begin{align*}  
\E(Y) &= (1+o(1))\, 
  \sqrt{\frac{\pi}{2n}}(r-1) \left((k-1) (r-1) \left( \frac{rk-k-r}{rk-k}\right)^{(r-1)(rk-r-k)/r} \right)^{n/(r-1)}  \\
& \hspace*{8cm} \times \exp \left(
 \frac{(r-1) \left(rk-r -2\right) }{2 (rk-r-k)} \right),
\end{align*}
see~\cite[Remark~6.4 and Corollary~2.3]{AGIR17}.
Using Corollary~\ref{PROBincEXC} we confirm this conjecture and
extend it to the case that $k$ and $r$ grow sufficiently slowly with $n \rightarrow \infty$. 

\begin{corollary}
For $n\geq 3$, let $r=r(n) \geq 3$, $k=k(n) \geq 2$ and assume that  $r$ divides $\Mk$ and $r-1$ divides $n$ for infinitely many values of $n$.  
Let $\mathcal{H}$ be chosen uniformly at random from $\Hrkr$ and
let $Y$ be the number of loose Hamilton cycles in $\mathcal{H}$. 
If $r^4 \, k^2=o(n)$ then
\begin{align*}
 \E(Y) &= 
  \sqrt{\frac{\pi}{2n}}(r-1) \left((k-1) (r-1) \left( \frac{rk-k-r}{rk-k}\right)^{(r-1)(rk-r-k)/r} \right)^{n/(r-1)}  \\
& \qquad \quad \times \exp \left(
 \frac{(r-1) \left(rk-r -2\right) }{2 (rk-r-k)} +O\left(\frac{r^4 k^{2}}{n}  \right)
\right).
 \end{align*}
\end{corollary}

\begin{proof}
Let $X$ be a fixed loose Hamilton cycle with $t=\nfrac{n}{r-1}$ edges and degree sequence $\xvec$. 
With $\kvec=(k,\ldots, k)$, we have
 \[
 \Mkx= (rk-r-k) \,t, \qquad 
\Mtkx = (k-2)_2  \,t +  (r-2)(k-1)_2 \,t.
 \]   
By Corollary~\ref{PROBincEXC}, the probability that a random hypergraph in $\Hrkr$ contains $X$ is
\[ 
\frac{(kn/r)_{t} \, r!^{t}\, \prod_{i=1}^{n} (k)_{x_i}  }{(kn)_{rt}}  \exp\left(
    \frac{r-1}{2}\,\left(\frac{\Mtk}{\Mk}-  \frac{\Mtkx}{\Mkx} \right)+ O\left( \beta \right)\right),
\]
where 
\begin{align*}
\beta &=
\frac{r^4 k^3}{(rk-r-k)\,t} + \frac{k^3}{(rk-r-k)^2\, t} + \frac{r\, k^4}{(rk-r-k)^3\, t^2}
=O\left( \frac{r^4 k^2}{n}\right),
\end{align*}
using the fact that $rk/(rk-r-k) = O(1)$.
Next, we have
 \begin{align*}
 \dfrac{1}{2}\,\left(\frac{\Mtk}{\Mk}-  \frac{\Mtkx}{\Mkx} \right) = \frac{(r-1) \left(rk-r -2\right) }{2 (rk-r-k)}.
 \end{align*}
The number of loose Hamilton cycles in the complete $r$-uniform hypergraph on $n$ vertices is 
\[\frac{( r-1)\,\, n! }{2n \, (r-2)!^t}.
\]
Therefore, by symmetry,
  \begin{align}\label{EHam}
 \E(Y) &=   \frac{( r-1)\,\, n! }{2n \, (r-2)!^t} \, \frac{(kn/r)! \, r!^{t}\, (kn-rt)!\, 
\prod_{i=1}^{n} (k)_{x_i}  }{(kn/r-t)!\, (kn)!}  \nonumber \\
 & \hspace*{3cm} \times \exp \left(
 \frac{(r-1) \left(rk-r -2\right) }{2 (rk-k-r)} +O\left(\frac{r^4 k^2}{n} \right)
\right).
 \end{align}
 Observe that $\prod_{i=1}^{n} (k)_{x_i}= \, \left(k(k-1)\right)^t \, k^{(r-2)t}$. 
 The factor outside the exponential can be estimated using Stirling's formula, giving
 \begin{align*}
 &\frac{(r-1) \,  n! \,  (k(k-1))^t \,\, k^{(r-2)t}
   r!^{t}\, (kt(r-1)-rt)! \,   ((kt(r-1)/r)! }{2n \, (r-2)!^t \,  (kt(r-1)/r-t)! \,\, (kt(r-1))!} 
 \\[12pt]
&= \sqrt{\frac{\pi}{2n}}(r-1) \left((k-1) \, (r-1) \,\,\left( \frac{rk-k-r}{rk-k}\right)^{(r-1)(rk-k-r)/r} \right)^{t} \,
  \exp\left(O\left(\frac{1}{n} +\frac{r}{kn-rt}\right)\right).
 \end{align*}
 The proof is completed by combining this expression with~(\ref{EHam}), since the error term from (\ref{EHam}) dominates.

 \end{proof}

\subsubsection*{Acknowledgments}
We are grateful to Brendan McKay and the anonymous referees for their helpful suggestions which have simplified our
argument.

\end{document}